\documentclass{amsart}

 \usepackage{fullpage}

 \usepackage{amssymb,latexsym,amsmath,amsthm}

\usepackage[numbers,sort&compress]{natbib}

 \renewcommand{\div}{\mathop{\mathrm{div}}\nolimits}

\newtheorem*{thm*}{Theorem A}

\newtheorem{thm}{Theorem}[section]

\newtheorem{dfn}{Definition}[section]

\newtheorem{lemma}{Lemma}[section]

\newtheorem{remark}{Remark}[section]

\newtheorem{cor}{Corollary}[section]

\newtheorem{op}{Open Problem}

\numberwithin{equation}{section}

\begin{document}

\author{Mostafa Fazly}

\address{Department of Mathematics, The University of Texas at San Antonio, San Antonio, TX 78249, USA}
\email{mostafa.fazly@utsa.edu}

\author{Henrik Shahgholian}
\address{Department of Mathematics, KTH, 
Lindstedtsv\"{a}gen 25,  100 44 Stockholm, Sweden}
\email{henriksh@math.kth.se}

\thanks{The first author is partially supported by  University of Texas at San Antonio Start-up Grant. The second author was partially supported by Swedish Research Council.}

\def\R{{\mathbb R} }
\def\p{$({\mathcal P})$}

\def\IR{{\mathbb{R}}}

\def\RR{{\mathbb{R}^n}}

\def\bg{\begin{equation*}}

\def\ng{\end{equation*}}

\def\bge{\begin{eqnarray*}}

\def\nge{\end{eqnarray*}}

\def\r{\R^{n+1}_{+}}
\def\rn{\R^{n}}
\def\br{\partial\r}
\def\super{\overline}

\title{Monotonicity formulas for coupled elliptic gradient systems with applications}
\maketitle

\begin{abstract}  
Consider the following coupled elliptic system of equations 
 \begin{equation*} \label{}
   (-\Delta)^s u_i    =   (u^2_1+\cdots+u^2_m)^{\frac{p-1}{2}} u_i  \quad  \text{in} \ \  \mathbb{R}^n , 
  \end{equation*}   
 where $0<s\le 2$,  $p>1$, $m\ge1$, $u=(u_i)_{i=1}^m$ and $u_i:\IR^n\to \IR$.  The qualitative behavior of solutions of the above  system has been studied from various perspectives in the literature including the free boundary problems and the classification of solutions. For the case of local scalar  equation, that is when $m=1$ and $s=1$, Gidas and Spruck in \cite{gs} and later Caffarelli, Gidas and Spruck in \cite{cgs}  provided the classification of solutions for Sobolev sub-critical and critical exponents. More recently, for the case of local system of equations that is when $m\ge1$ and $s=1$ a similar classification result is given by  Druet, Hebey and   V\'{e}tois in \cite{dhv} and references therein. In this paper,  we derive monotonicity formulae for entire solutions of the above local, when $s=1,2$, and nonlocal, when $0<s<1$ and $1<s<2$, system. These monotonicity formulae are of great interests due to the fact that a counterpart of the celebrated monotonicity formula of Alt-Caffarelli-Friedman \cite{acf} seems to be challenging to derive for system of equations.   Then,  we apply these formulae to give a classification of finite Morse index solutions.  In the end, we provide an open problem in regards to monotonicity formulae for  Lane-Emden systems. 
\end{abstract}


\section{Introduction}

\subsection{Background}

Our objective in this paper is to establish monotonicity formulae for solutions of the following 
  coupled elliptic system of equations
 \begin{equation} \label{main}
   (-\Delta)^s u_i    =   |u|^{p-1} u_i  \quad  \text{in} \ \  \mathbb{R}^n 
  \end{equation}   
 and classify  finite Morse index solutions when $0<s\le 2$,  $p>1$, $m\ge1$, $u=(u_i)_{i=1}^m$ and $u_i:\IR^n\to \IR$.  Note that when $m=1$ the above is known as the Lane-Emden equation. 
    An important feature of   system (\ref{main}) is the variational structure of the coupled equations of the form
     \begin{equation}\label{H}
  -\Delta u= \nabla H(u)   \quad  \text{in} \ \  \mathbb{R}^n. 
  \end{equation}
Note that   (\ref{main}) is a particular case of  (\ref{H}) for $H(u)=\frac{1}{p+1} |u|^{p+1}$.    Due to this variational structure,  qualitative behaviour of solutions of system (\ref{H}) has been studied extensively in the context of elliptic partial differential equations from both pure and applied mathematics perspectives.    Let us mention that Andersson et al. in \cite{asuw} considered system (\ref{main}), up to a negative sign, when $p=0$ and $s=1$ which reads
   \begin{equation} \label{}
   \Delta u_i    =  \frac{u_i}{|u|} \chi_{\{|u|>0\}}  . 
  \end{equation} 
This is closely related to minimizers of the energy 
   \begin{equation} \label{}
   \int_{} \left(\sum_{i=1}^m |\nabla u_i|^2 + 2 |u|  \right) dx, 
  \end{equation} 
and they studied the regularity of free boundaries.  To prove regularity results in \cite{asuw}, they established a monotonicity formula that is inspired by the one given by Weiss in \cite{weiss,weiss2}. We also refer interested readers to  \cite{an,csy} for regularity results on cooperative systems and  to the book of   Petrosyan, Shahgholian and Uraltseva in \cite{psu} for more information.    It should be remarked that  as a particular case, one can consider the case of two components, that is $m=2$, and $p=3$  namely 
 \begin{eqnarray}\label{twosys}
 \left\{ \begin{array}{lcl}
\hfill (- \Delta)^s  u &=& ( u^2+ v^2) u   \ \ \text{in}\ \ \mathbb{R}^n,\\   
\hfill (- \Delta)^s v &=& ( u^2+ v^2) v     \ \ \text{in}\ \ \mathbb{R}^n.
\end{array}\right.
  \end{eqnarray}
  Note that the above system is a special case of the nonlinear Schr\"{o}dinger system that is 
  \begin{eqnarray}\label{sch} 
  \left\{ \begin{array}{lcl} \hfill (- \Delta)^s  u &=& (\mu_1 u^2+\beta v^2) u   \ \ \text{in}\ \ \mathbb{R}^n,\\    \hfill (- \Delta)^s v &=& (\mu_2 u^2+\beta v^2) v     \ \ \text{in}\ \ \mathbb{R}^n, 
  \end{array}\right.  \end{eqnarray} where $\mu_1,\mu_2,\beta$ are constants. We would like to mention that most of the results in the paper are valid for (\ref{sch}) as well.   We refer interested readers to \cite{ww,tvz,nttv} and references therein for more information regarding classification of solutions of system (\ref{sch}) for various parameters $\mu_1,\mu_2$ and $\beta$.

\subsection{Tools and methodology (Historical development) }
Semilinear elliptic equations (with almost four decades of history)  are probably the most widely and intensely studied  equations in PDE. The simplest   equation in this class is  expressed as 
\begin{equation}
\Delta u = f(u), 
\end{equation}
and exhibits as many possible features, as the right hand side nonlinearity $f(u)$ may do.  This equation  has also  been studied from so many different  perspectives, that  any attempt to try to list  them here would probably fail. There are however two different types of questions,  of interest to the  current authors,  that seem to be interconnected and developed in parallel, but almost separately: 

\vspace{3mm}

\noindent
{\bf i)} Analysis of the singular set: $\{x: \ \nabla u =0 \}$.\newline
{\bf ii)} Stability and Liouville's type questions.

One can actually add a third less studied, and probably more complicated  problem to the above, which is: 

\noindent
{\bf iii)} Structure of the blow-up set: $\{x: \   |u(x)| =\infty \}$.

\vspace{3mm}

In studying  the  above problems,  experts  have (in many cases) developed parallel tools and ideas to handle technical passages of the analysis of questions in each cases. Two of these   tools, also utilized in this paper, are  
 {\it Monotonicity functional}, and   {\it Blow-up/down Analysis}.

The particular  monotonicity functional used here  combines the energy functional along with a balanced term (see below for explicit form). The use of monotonicity functional  (in the way we present here) can be traced back\footnote{It is worth mentioning that a different type of almost monotonicity functional was used by Arne Beurling in his thesis:
  A. Beurling, Etudes sur un probleme de majoration, thesis, (Uppsala), 1933. The monotonicity functional of A. Beurling states that  for a jordan curve $\gamma$, the product of the harmonic measures for both sides of the curve can be controlled as follows: 
  $\omega_1(B_r (z) \cap \gamma ) \cdot \omega_2(B_r(z) \cap \gamma )  \leq  Ar^2$, where $z\in \gamma$, and $\omega_1, \omega_2$ denote the harmonic measures on each side of the curve.   }
   to the work of  Fleming in  \cite{Flem-62} for area minimizing currents which has subsequently been proved by Allard in \cite{all} for stationary rectifiable $n-$varifolds, that was later developed further by others. The monotonicity  functional that we shall use in this paper originates in the study of harmonic maps  by Price in \cite{price} and Schoen and Uhlenbeck in \cite{su}, see also \cite{s}.  And in connection to free boundary problems,   it  was first used by  Ou \cite{Ou-94}, and developed later by Weiss in \cite{weiss2}  and others in  various forms.    In the context of semilinear heat equations, a similar monotonicity formula is given by Giga and Kohn in \cite{gk} and for the corresponding elliptic equation by Pacard \cite{pac}.

The second main tool,  blow-up/down analysis, has its origin in local regularity theory and the so-called linearization technique  (also called harmonic blow-up).  To study local structure of level surfaces of a solution to the above semilinear problem, one considers scaled version of the problem and classifies  the limit manifold of 
such scalings.  The limit manifold naturally carries information of the local behavior of the solutions and hence one can then with some devices link back this behavior to the local problem and deduce  the expected  properties.
It is now that the role of monotonicity functional become crucial and indispensable in that one can use the fact that 
when blowing up a solution, the monotonicity functional, call it $E(r,u)$,  being monotone will have a limit. At the same time this functional has a nice scaling properties $E(rs,u)= E(s,u_r)$  for $r,s >0$ and $u_r$ a correctly scaled version of $u$. Hence one obtains 
\begin{equation}
E(0^+, u)= \lim_{r\to 0} E(rs,u) =  \lim_{r\to 0} E(s,u_r) = E(s,u_0),
\end{equation}
so that $E(s,u_0) $ is constant. Now a strong version of the monotonicity functional asserts that the only time $E$ is constant is when the function $u_0$ is homogeneous, where the order of homogeneity is dictated by $E$.
Since the same can be done by blowing-down, we will have that $u_0, u_\infty$ are homogeneous of same order.
Next if one can prove that the   homogenous solutions are unique we must then have $u_0= u_\infty$ and hence 
$E(0^+,u) = E(\infty, u)$, and we arrive at  $E(r,u) = constant $.  Therefore $u$ is itself homogeneous, and already classified. 

In the regularity theory of level surfaces the classification of homogeneous global solutions is one of the key 
elements of the theory, and many time a very hard nut to crack,  in higher dimensions; in two dimensions  homogenous solutions can be computed directly. In the stability theory and Liouville type problems this part is based on computations of eigenvalues of the corresponding Laplace-Beltrami on the sphere, as also done in this paper. As shown  here below  (and in many preceding papers by experts in this area) the only homogenous solutions in the {\it appropriate}  space are the trivial  solutions. In particular,  the energy functional becomes zero
and hence $E(r,u)=0$.  From here one deduces that $u\equiv 0 $ is the only solution satisfying the given condition.

It is worthwhile remarking the fundamental feature of these problems, that in some cases depend on the dimension and order of homogeneity are directly a consequence of the eigenvalues of Laplace-Beltrami, where in 
the case of index-theory results in the computation of  eigenvalues for nonlinear Laplace-Beltrami, that in turn implies that non-trivial solutions do not exists in certain range of values. These ideas have their origin in the regularity theory of minimal surfaces, that boils down to proving minimal cones do not exists in dimensions less than eight.


\subsection{Problem setting}
 Our main technique is to derive monotonicity formulae for solutions of (\ref{main}) for various values of parameter $0<s\le 2$. To provide such monotonicity formulae we consider various cases.

\noindent {\bf Case $s=1$.} Consider the following functional  for every $\lambda>0$ and $x_0\in\RR$
\begin{equation}
 E_1(u,\lambda,x_0)  :=  \lambda^{-n+2\frac{p+1}{p-1}} \int_{B_\lambda(x_0)} \left(\frac{1}{2} \sum_{i=1}^m |\nabla u_i|^2 - \frac{1}{p+1} |u|^{p+1}\right ) + \frac{1}{p-1}  \lambda^{-n+2\frac{p+1}{p-1}-1} \int_{\partial B_\lambda(x_0)} |u|^2, 
\end{equation}
then the following monotonicity formula holds for classical solutions of (\ref{main}). 
 \begin{thm}\label{thmonos1}
  Suppose that $u$ is a solution of (\ref{main}) for $s=1$.  Then,  $E_1$ is a nondecreasing function of $\lambda$ and in fact 
\begin{equation}
\frac{dE_1}{d\lambda} = \lambda^{-n+1 + 2\frac{p+1}{p-1}}  \int_{\partial B_\lambda(x_0)} \sum_{i=1}^m \left( \frac{\partial u_i}{\partial r} + \frac{2}{p-1} \frac{u_i}{r} \right)^2  ,
\end{equation}
where $\frac{\partial}{\partial r}$ is polar derivative. 
\end{thm}
Note that for the case of single equations, that is when $m=1$, similar monotonicity formulae are given by Pacard in \cite{pac} and Weiss in \cite{weiss,weiss2}.  For the case of systems that is when $m\ge 1$, very recently, Andersson, Shahgholian,  Uraltseva and 
 Weiss in \cite{asuw} provided a monotonicity formula for solutions of system (\ref{main}) when $p=0$  and applied it to study free boundary problems.      Let us mention that very similar monotonicity formulae appear in the field of harmonic maps that is 
 \begin{equation}
 -\Delta u=|\nabla u|^2 u, 
 \end{equation}
 where $u:\mathbb R^n\to \mathbb S^{m-1}$,  see Evans  in \cite{evans,evans2} and reference therein. 

\noindent {\bf Case $s=2$.} For every $\lambda>0$ and $x_0\in\RR$ define 
\begin{eqnarray}\label{energyE2}
\ \  \ \ \ \   E_2(u, \lambda, x_0)& :=& \lambda^{4\frac{p+1}{p-1}-n}   \int_{ B_\lambda(x_0)} \left( \frac{1}{2} \sum_{i=1}^m |\Delta u_i|^2-  \frac{1}{p+1}  |u|^{p+1}   \right)
 \\&& \nonumber - \frac{4}{p-1} \left( \frac{p+3}{p-1}-n\right) \lambda^{1+\frac{8}{p-1}-n}  \int_{  \partial B_\lambda(x_0)} |u|^2 \\
&&\nonumber -\frac{4}{p-1} \left( \frac{p+3}{p-1}-n\right) \frac{d}{d\lambda} \left[ \lambda^{\frac{8}{p-1}+2-n}  \int_{ \partial B_\lambda(x_0)}  |u|^2 \right]\\
&& \nonumber + \frac{1}{2} \lambda^3   \frac{d}{d\lambda} \left[ \lambda^{\frac{8}{p-1}+1-n}  \int_{   \partial B_\lambda(x_0)} \sum_{i=1}^m\left(  \frac{4}{p-1} \lambda^{-1} u_i+ \frac{\partial u_i}{\partial r}\right)^2 \right]\\
&&\nonumber + \frac{1}{2}    \frac{d}{d\lambda} \left[ \lambda^{4\frac{p+1}{p-1}-n}  \int_{  \partial B_\lambda(x_0)} \sum_{i=1}^m\left(  | \nabla u_i|^2 - \left|\frac{\partial u_i}{\partial r}\right|^2 \right) \right]\\
&&\nonumber + \frac{1}{2}    \lambda^{4\frac{p+1}{p-1}-n-1}  \int_{ \partial B_\lambda(x_0)} \sum_{i=1}^m \left(  | \nabla u_i|^2 - \left|\frac{\partial u_i}{\partial r}\right|^2 \right).
\end{eqnarray}
 Then the following monotonicity formula holds for the fourth order Lane-Emden system. 
 \begin{thm}\label{thmonos2} 
 Suppose that $n\ge 5$, $p > \frac{n+4}{n-4}$  and $u$ is a solution of (\ref{main}) when $s=2$.   For any $\lambda>0$ and $x_0\in\RR$ 
\begin{eqnarray}
\frac{dE_2(u, \lambda, x_0)}{d\lambda} \ge C   \lambda^{\frac{8}{p-1}+2-n}    \int_{  \partial B_\lambda(x_0)}  \sum_{i=1}^m \left(  \frac{4}{p-1} r^{-1} u_i+ \frac{\partial u_i}{\partial r}\right)^2 ,
\end{eqnarray}
where $E_2$ is defined by (\ref{energyE2}) and $C$ is independent from $\lambda$.
\end{thm}
 Note also that the above monotonicity formula is provided by Davila, Dupaigne, Wang and  Wei in \cite{ddww} for the case of single equations that is when $m=1$.  For the case of fractional Laplacian, we provide monotonicity formulae for the extension problems.  Assume that $u_i\in C^{2\sigma}(\mathbb R^n)$, $\sigma >s>0$ and
\begin{equation}
 \int_{\mathbb R^n} \frac{|u_i(y)|}{(1+|y|)^{n+2s}} dy<\infty,
 \end{equation}
 for each $1\le i\le m$.   The fractional Laplacian of $u_i$ when $0<s<1$ denoted by 
\begin{equation}
(-\Delta)^s u_i(x):= p.v. \int_{\mathbb R^n} \frac{u_i(x)-u_i(y)}{|x-y|^{n+2s}} dy ,
\end{equation}
is well-defined for every $x\in\mathbb R^n$. Here $p.v.$ stands for the principle value. It is by now standard that the fractional Laplacian can be seen as a Dirichlet-to-Neumann operator for a
degenerate but local diffusion operator in the higher-dimensional half-space $\mathbb R^{
n+1}$, see Caffarelli and Silvestre in \cite{cs}.  In other words, for $u_i\in C^{2\sigma} \cap L^1(\mathbb R^n, (1+|y|^{n+2s})dy)$ when $\sigma>s$ and $0<s<1$,  there exists $v=(v_i)_{i=1}^m$ such that $v_i\in C^2(\mathbb R^{n+1}_+)\cap C(\super{\mathbb R^{n+1}_+})$,  $y^{1-2s}\partial_{y} v_i\in C(\super{\mathbb R^{n+1}_+})$ and 

\begin{equation}\label{main1e}
\left\{
\begin{aligned}
\nabla\cdot(y^{1-2s}\nabla  v_i)&=0&\quad\text{in $\R^{n+1}_{+}$,}\\
 v_i&= u_i&\quad\text{on $\br$,}\\
-\lim_{y\to0} y^{1-2s}\partial_{y} v_i&= \kappa_s   \vert  v\vert^{p-1} v_i  &\quad\text{on $\br$,}
\end{aligned}
\right.
\end{equation}
for the following constant $\kappa_s$, 
\begin{equation} \label{kappas}
\kappa_s := \frac{\Gamma(1-s)}{2^{2s-1} \Gamma(s)}.
\end{equation}

For the case of $1<s<2$,  there are various definitions for the fractional operator $(-\Delta)^{s}$, see  \cite{fw,yang,cc,cg}.  From the Fourier transform one can define the fractional Laplacian by
\begin{equation} 
\widehat{ (-\Delta)^{s}}u_i(\zeta)=|\zeta|^{2s} \hat u_i (\zeta)  ,
\end{equation}
or inductively by  $(-\Delta)^{s}=  (-\Delta)^{s-1} o (-\Delta)$.  Note that Yang in \cite{yang} gave a counterpart of the extension problem by  Caffarelli and Silvestre in \cite{cs} for  the fractional Laplacian $(-\Delta)^{s}$, where $s$ is  any positive, noninteger number.   In other words,  he showed that  the higher order fractional Laplacian operator can also be regarded as the Dirichlet-to-Neumann map for an extension function  satisfying a  higher order elliptic equation in the upper half space with one extra spatial dimension. More precisely, there exists an extension function $v_i\in W^{2,2}(\mathbb R^{n+1}_+,y^b)$ such that 
\begin{eqnarray}\label{main2e}
 \left\{ \begin{array}{lcl}
\hfill \Delta^2_b v_i&=& 0   \ \ \text{in}\ \ \mathbb{R}^{n+1}_{+},\\
\hfill  \lim_{y\to 0} y^{b}\partial_y{v_i}&=& 0   \ \ \text{in}\ \ \partial\mathbb{R}^{n+1}_{+},\\
\hfill \lim_{y\to 0} y^{b} \partial_y \Delta_b v_i &=& C_{n,s} |v|^{p-1} v_i  \ \ \text{in}\ \ \mathbb{R}^{n}, 
\end{array}\right.
\end{eqnarray}
  where $b:=3-2s$ and $\Delta_b v_i:=y^{-b} \div(y^b \nabla v_i)$.   We refer interested readers to  Case and Chang \cite{cc} and Chang and Gonzales \cite{cg} as well.  We are now ready to provide  monotonicity formulae for various parameters $s>0$.

 \noindent {\bf Case $0<s<1$.}  Let $v=(v_i)_{i=1}^m$ be a solution of the  extension problem (\ref{main1e}).    Now define the energy functional  for any  $\lambda>0$ and $x_{0}\in\br$  as 
 \begin{eqnarray}\label{energyEs1}
  E_s( v,\lambda,x_0) &:=&   \lambda^{\frac{2s(p+1)}{p-1}-n}\left(\frac12\int_{\r\cap B_\lambda} y^{1-2s} \sum_{i=1}^m \vert\nabla v_i \vert^2\;dx\,dy - \frac{\kappa_{s}}{p+1}\int_{\br\cap B_\lambda}  \vert  v \vert^{p+1}\;dx\right)\\
&&\nonumber + \lambda^{\frac{2s(p+1)}{p-1}-n-1}\frac{s   }{p+1}\int_{\partial B_\lambda \cap\r}y^{1-2s} \vert v \vert ^2\;d\sigma.
\end{eqnarray}
We are now ready to provide a monotonicity formula for  solutions of \eqref{main1e}  when $0<s<1$. 
 \begin{thm}\label{thmono1s}
Suppose that $0<s<1$.  Let $ v=(v_i)_{i=1}^m$ where each $v_i\in C^2(\r)\cap C(\super\r)$ be  a solution of \eqref{main1e}  such that $y^{1-2s}\partial_{y} v_i\in C(\super\r)$. Then, $E_s$ is a nondecreasing function of $\lambda$. Furthermore,
\begin{equation}
\frac{dE_s}{d\lambda} = \lambda^{\frac{2s(p+1)}{p-1}-n+1}\int_{\partial B(x_{0},\lambda)\cap\r}y^{1-2s}\sum_{i=1}^m\left(\frac{\partial  v_i}{\partial r}+\frac{2s}{p-1}\frac { v_i}r\right)^2\;d\sigma ,
\end{equation}
where $E_s$ provided in (\ref{energyEs1}). 
\end{thm}

 \noindent {\bf Case $1<s<2$.}  Suppose that  $v=(v_i)_{i=1}^m$ is  a solution of the  extension problem (\ref{main2e}).   Similarly, now define the energy functional  for any  $\lambda>0$ and $x_{0}\in\br$  as 
 \begin{eqnarray}\label{energyEs}
E_s(v,\lambda,x_0)
 & :=&  \lambda^{2s\frac{p+1}{p-1}-n} \left(   \int_{  \mathbb{R}^{n+1}_{+}\cap B_\lambda(x_0)} \frac{1}{2} y^{3-2s}\sum_{i=1}^m |\Delta_b v_i|^2-  \frac{C_{n,s}}{p+1} \int_{  \partial\mathbb{R}^{n+1}_{+}\cap B_\lambda(x_0)} |v|^{p+1}   \right)
\\
&& \nonumber - \frac{s}{p-1} \left( \frac{p+2s-1}{p-1}-n\right) \lambda^{-3+2s+\frac{4s}{p-1}-n}  \int_{  \mathbb{R}^{n+1}_{+}\cap \partial B_\lambda(x_0)} y^{3-2s} |v|^2 \\
&& \nonumber -\frac{s}{p-1} \left( \frac{p+2s-1}{p-1}-n\right) \frac{d}{d\lambda} \left[ \lambda^{\frac{4s}{p-1}+2s-2-n}  \int_{  \mathbb{R}^{n+1}_{+}\cap \partial B_\lambda(x_0)} y^{3-2s} |v|^2 \right]\\
&& \nonumber + \frac{1}{2} r^3   \frac{d}{d\lambda} \left[ \lambda^{\frac{4s}{p-1}+2s-3-n}  \int_{  \mathbb{R}^{n+1}_{+}\cap \partial B_\lambda(x_0)} y^{3-2s}   \sum_{i=1}^m \left(  \frac{2s}{p-1} \lambda^{-1} v_i+ \frac{\partial v_i}{\partial r}\right)^2 \right]\\
&&\nonumber + \frac{1}{2}    \frac{d}{d\lambda} \left[ \lambda^{2s\frac{p+1}{p-1}-n}  \int_{  \mathbb{R}^{n+1}_{+}\cap \partial B_\lambda(x_0)} y^{3-2s} \sum_{i=1}^m  \left(  | \nabla v_i|^2 - \left|\frac{\partial v_i}{\partial r}\right|^2 \right) \right]\\
&&\nonumber + \frac{1}{2}    \lambda^{2s\frac{p+1}{p-1}-n-1}  \int_{  \mathbb{R}^{n+1}_{+}\cap \partial B_\lambda(x_0)} y^{3-2s} \sum_{i=1}^m   \left(  | \nabla v_i|^2 - \left|\frac{\partial v_i}{\partial r}\right|^2 \right) .
\end{eqnarray}
 Considering the above energy functional,  we now provide a monotonicity formula for  solutions of \eqref{main2e}  when $1<s<2$. 
\begin{thm}\label{thmono2s}
Assume that $n>2s$ and $p>\frac{n+2s}{n-2s}$. Let $v=(v_i)_{i=1}^m$ be a solution of (\ref{main2e}) then $E(v,\lambda,x_0)$ is a nondecreasing function of $\lambda>0$. In addition, 
\begin{equation}
\frac{dE_s(v,\lambda,x_0)}{d\lambda} \ge C(n,s,p) \  \lambda^{\frac{4s}{p-1}+2s-2-n}    \int_{  \mathbb{R}^{n+1}_{+}\cap \partial B_\lambda(x_0)}  y^{3-2s}  \sum_{i=1}^m  \left(  \frac{2s}{p-1} r^{-1} v_i + \frac{\partial v_i}{\partial r}\right)^2 ,
\end{equation}
where $E_s$ is given by (\ref{energyEs}) and  $C(n,s,p)$ is independent from $\lambda$.
\end{thm}

 We also refer interested readers to Davila, Dupaigne and Wei in \cite{ddw} and  to Wei and the first author in \cite{fw} for a similar monotonicity formula for the case of scalar equations and $0<s<1$ and $1<s<2$, respectively. Before we state our main results let us present some backgrounds regarding classification of solutions of (\ref{main}) in the absence of stability.  We provide such classifications for scalar equations and systems separately.

\begin{remark} Note that monotonicity formulae provided as Theorems \ref{thmonos1}-\ref{thmono2s} hold for the following system with a slightly more general right-hand side, 
 \begin{equation} \label{mainlambda}
   (-\Delta)^s u_i    =   |u|^{p-1} \left(\alpha_i u_i^+ +   \beta_i u_i^- \right) \quad  \text{in} \ \  \mathbb{R}^n,  
  \end{equation} 
where $\alpha_i$ and $\beta_i$ are positive constants for $1\le i\le m$.  
\end{remark}
 
Suppose that $m=1$ when  (\ref{main}) turns into a single equation.  We first consider the local operator cases meaning $s=1$ and $s=2$.  Suppose that $s=1$ and the parameter $p$ is in the subcritical case that is when 
$1<p<p_S(n,1)$ where
\begin{eqnarray}\label{ps1}
p_S(n,1):= \left\{ \begin{array}{lcl}
\hfill \infty   \ \ && \text{if } n \le 2,\\
\hfill  \frac{n+2}{n-2} \ \ && \text{if } n > 2. 
\end{array}\right.
\end{eqnarray}
For this case, there is a very well-known classification result of Gidas and Spruck in \cite{gs} stating that the only
nonnegative solution of the Lane-Emden equation (\ref{main}) with $s=1$ is the trivial solution. For the critical case, that is when $p=p_S(n,1)$,  Caffarelli, Gidas and Spruck in \cite{cgs}  that there is a unique (up to translation and rescaling) positive solution for the Lane-Emden equation.  For the fourth order Lane-Emden equation, that is when $s=2$, Wei and Xu in \cite{wx} provided a similar classification result for the subcritical $1<p<p_S(n,2)$ and the critical cases $p=p_S(n,2)$ when 
\begin{eqnarray}\label{ps2}
p_S(n,2):= \left\{ \begin{array}{lcl}
\hfill \infty   \ \ && \text{if } n \le 4,\\
\hfill  \frac{n+4}{n-4} \ \ && \text{if } n > 4, 
\end{array}\right.
\end{eqnarray}
see also \cite{lin}.    Note that for the fractional Laplacian operator $0<s<1$, such a classification result is given by   Li \cite{li} and Chen-Li-Ou \cite{clo} where the critical exponent is 
\begin{eqnarray}\label{pns}
p_S(n,s)= \left\{ \begin{array}{lcl}
\hfill \infty   \ \ && \text{if } n \le 2s,\\
\hfill  \frac{n+2s}{n-2s} \ \ && \text{if } n > 2s.
\end{array}\right.
\end{eqnarray}
For the case of system of equations, that is when  $m\ge 2$,   Druet, Hebey and   V\'{e}tois in \cite{dhv} provided a classification result for solutions of (\ref{main}) where $s=1$ via the moving sphere method. Suppose that $p=\frac{n+2}{n-2}$ and $u=(u_i)_{i=1}^m$ is a nonnegative classical solution of (\ref{main}) where $s=1$. Then they proved that there exist $x_0\in\RR, \lambda>0$,  $ \Lambda\in \mathcal S^{m-1}_+$ such that 
\begin{equation}
u(x)= \left(   \frac{\lambda}{\lambda^2+\frac{1}{n(n-2)}|x-x_0|^2} \right)^{\frac{n-2}{2}} \Lambda.
\end{equation} 
We also refer interested readers to \cite{dh,ht,h} where the authors studied various counterparts of  system (\ref{main}).   Note  that the following singular function 
\begin{equation}\label{singsol}
u_s(x) = \mathcal A  |x|^{-\frac{2s}{p-1}} 
\ \ \text{for} \ \   \mathcal A\in \mathbb R^m \ \ \text{with} \ \   
| \mathcal A|^{p-1} =2^{2s} \frac{\Gamma(\frac{n}{2}-\frac{s}{p-1}) \Gamma(s+\frac{s}{p-1})}{\Gamma(\frac{s}{p-1}) \Gamma(\frac{n-2s}{2}-\frac{s}{p-1})} ,
  \end{equation}  
solves (\ref{main}) in $\mathbb R^n\setminus\{0\}$ for a supercritical parameter $p$ that is $p>p_S(n, s)$.  Before we state our main results, let us  define the notion of stable solutions.  

     \begin{dfn}
  We say a solution $u$ of (\ref{main}) is  stable outside a compact set if there exists $R_0>0$ such that 
\begin{equation} \label{stability}
\sum_{i=1}^m\int |u|^{p-1} \phi_i^2 +(p-1)\sum_{i,j=1}^{m} \int |u|^{p-3}u_iu_j \phi_i \phi_j \le 
\sum_{i=1}^{m} ||\phi_i||^2_{\dot H^s(\RR)},
\end{equation} 
for any $\phi=(\phi_i)_{i=1}^m$ where $ \phi_i \in C_c^\infty(\mathbb R^n\setminus \overline {B_{R_0}}) $ for $1\le i\le m$. 
  \end{dfn}
  
  Here is our main result.  
  
  \begin{thm}\label{mainthm}
Suppose $0<s\le 2$ and $n> 2s$. Let $u$ be a solution of (\ref{main}) that is stable outside a compact set. Then either for $1<p<p_S(n,s)$ or for $p>p_S(n,s)$ and 
\begin{equation}\label{conditionp} 
p \frac{\Gamma(\frac{n}{2}-\frac{s}{p-1}) \Gamma(s+\frac{s}{p-1})}{\Gamma(\frac{s}{p-1}) \Gamma(\frac{n-2s}{2}-\frac{s}{p-1})}>\frac{ \Gamma(\frac{n+2s}{4})^2 }{\Gamma(\frac{n-2s}{4})^2},
\end{equation}
each component $u_i$ must be identically zero.  For the case of Sobolev critical exponent, that is when $p=p_S(n,s)$, a solution $u$ has finite energy that is
\begin{equation}
  ||u||^{p+1}_{L^{p+1}(\RR)} = ||u||^2_{\dot H^s(\RR)}< \infty. \end{equation}  
In this case,  if in addition $u$ is stable,  then each component $u_i$ must be identically zero.
\end{thm}
As a direct consequence,  the above theorem implies that the only nonnegative solution for system (\ref{twosys}) when $s=1$ is the trivial solution for dimensions $n < 12$ and $n\neq 4$.     Here,  is how this article is structured.  For the rest of this article,  we provide a proof for Theorem \ref{mainthm} considering  various cases for parameter $s>0$.   In Section \ref{ls1},  we consider the case $s=1$ that is when the operator in (\ref{main}) is the local Laplacian operator. In Section \ref{ls2}, we let $s=2$ that refers to the bi-Laplacian operator.  Lastly, in Section \ref{nls},  we consider nonlocal cases $0<s<1$ and $1<s<2$. For these non-integer parameters, the operator in (\ref{main})  is a fractional Laplacian operator. For all cases $0<s\le 2$, we  apply blow-down analysis arguments  and monotonicity formulae. 

  \section{Local Case: Laplacian operator}\label{ls1}
 In this section we assume that $s=1$. Therefore system (\ref{main}) turns into the following form
  \begin{equation} \label{main1}
   -\Delta u_i    =   |u|^{p-1} u_i  \quad  \text{in} \ \  \mathbb{R}^n. 
  \end{equation}  
  Note that this is a particular case of system (\ref{H}). For a general nonlinearity $H: \mathbb R^m\to \mathbb R$ such that $\nabla H\ge 0$, it is proved in \cite{fg} that  bounded stable solutions of (\ref{H}) are constant when $n\le 4$. For radial solutions, this is known to hold in more dimensions that is when $n \le 9$ without any sign assumptions on the nonlinearity, see \cite{mf}.  In addition it is known that at least for the case of radial solutions the dimension $n=9$ is the optimal dimension. For the rest of this section we prove the following classification of finite Morse index solutions of  (\ref{main1}).  
  \begin{thm}\label{thm1stab} 
  Suppose that $u=(u_i)_{i=1}^m$ is a finite Morse index solution of (\ref{main1}) when $m\ge 1$ and $n\ge 3$. Let $1<p<\frac{n+2}{n-2}$ and  $\frac{n+2}{n-2}<p<p_c(n)$ where 
\begin{eqnarray}\label{pc1}
p_c(n)= \left\{ \begin{array}{lcl}
\hfill \infty   \ \ && \text{if } n \le 10,\\
\hfill  \frac{(n-2)^2 -4n+8\sqrt{n-1} }{(n-2)(n-10)} \ \ && \text{if } n \ge 11.
\end{array}\right.
\end{eqnarray}
Then each $u_i$ must be identically zero.  For the Sobolev critical exponent $p=\frac{n+2}{n-2}$, a solution $u$ has finite energy that is
  \begin{equation} 
   \int_{\mathbb R^n} |u|^{p+1}  = \sum_{i=1}^m \int_{\mathbb R^n} |\nabla u_i|^2< \infty. \end{equation}  
In this case,  if in addition $u$ is stable,  then each component $u_i$ must be identically zero.
\end{thm}

Note that for the case of scalar equations,  that is when $m=1$,  the above theorem is given by Farina in \cite{f1}. He used a Moser iteration type argument for the proof. We refer interested reader to Crandall and Rabinowitz in \cite {cr} for a similar approach.   To prove the above theorem, we apply a blow-down analysis argument as well as the monotonicity formula given as Theorem \ref{thmonos1}.  We now derive a few elliptic estimates.

\begin{lemma}\label{lmes}
Suppose that $u$ is a stable solution of (\ref{main}). Then,  for any $R>1$ 
\begin{equation}
\label{esp} \int_{B_R} |u|^{p+1} \le C R^{n-2\frac{p+1}{p-1}} \ \ \text{and}\ \  \sum_{i=1}^m \int_{B_R} |\nabla u_i|^{2} \le  C R^{n-2\frac{p+1}{p-1}} , 
\end{equation}
where $C$ is a positive constant that is independent from $R$. 
\end{lemma}
\begin{proof}
Test the stability inequality on $\phi_i=u_i \zeta_R$ where $\zeta_R: C_c^\infty(\mathbb R^n)\to \mathbb R$ and $\zeta_R\equiv 1$ on $B_{R}$ and $\zeta_R\equiv 0$ on $\mathbb R^n \setminus B_{2R}$ with $||\zeta_R||_{L^\infty(B_{2R}\setminus B_R)} \le C R^{-1}$. Now multiply both sides of (\ref{main}) with $u_i \zeta_R^2$ and integrate by parts. Equating the inequalities that one gets from this and from stability finishes the proof. 
\end{proof}
Applying the H\"{o}lder's inequality,   we get the following $L^2$ estimate. 
\begin{cor}
Suppose that $u$ is a stable solution of (\ref{main}). Then,  for any $R>1$
\begin{equation}\label{}
\int_{B_R} |u|^{2} \le C R^{n-\frac{4}{p-1}} , 
\end{equation}
where $C$ is a positive constant that is independent from $R$. 
\end{cor}
  
  In this part,  we present a classification of stable homogeneous solutions.  This is a key point in our proof of Theorem \ref{thm1stab}. 
  
\begin{thm}
Suppose that $u=(u_i)_{i=1}^m$ for $u_i= r^{-\frac{2}{p-1}} \psi_i(\theta)$ is a stable solution of (\ref{main}). Then,  each $\psi_i$ is identically zero provided $\frac{n+2}{n-2}<p<p_c(n)$ where $p_c(n)$ is given by (\ref{pc1}). 
\end{thm}
\begin{proof}
We omit the proof here, since a similar argument will be provided in the proof of Theorem \ref{thmhomo2} for the fourth order case. 

\end{proof}

\noindent {\bf Proof of Theorem \ref{thm1stab}}.  The proof is based on a blow-down analysis and it relies on the monotonicity formula provided as Theorem \ref{thmonos1}.  We omit the details,  since  similar arguments will be provided for the poof of Theorem \ref{thm2stab}.  For the case of Sobolev critical exponent, one can conclude the result via applying the Pohozaev indentity. 

\hfill $ \Box$

  \section{Local Case: Bi-Laplacian Operator}\label{ls2}
In this section,  we consider the following fourth order system 
   \begin{equation} \label{main2}
   \Delta^2 u_i    =   |u|^{p-1} u_i  \quad  \text{in} \ \  \mathbb{R}^n. 
  \end{equation}  
  This section is devoted to the proof of the following theorem. 

\begin{thm}\label{thm2stab}
Suppose that $u=(u_i)_{i=1}^m$ is a finite Morse index solution of (\ref{main2}) when $m\ge 1$ and $n\ge 5$. Let $1<p<\frac{n+4}{n-4}$ and  $\frac{n+4}{n-4}<p<\bar p_c(n)$ where 
\begin{eqnarray}\label{pc2}
\bar p_c(n)= \left\{ \begin{array}{lcl}
\hfill \infty   \ \ && \text{if } n \le 12,\\
\hfill  \frac{ n+2 -\sqrt{ n^2 +4 -n\sqrt{n^2-8n+32} }}{n-6-\sqrt{n^2+4-n\sqrt{n^2-8n+32}}} \ \ && \text{if } n \ge 13.
\end{array}\right.
\end{eqnarray}
Then,  each $u_i$ must be identically zero.  For the Sobolev critical exponent $p=\frac{n+4}{n-4}$, a solution $u$ has finite energy that is
\begin{equation}
 \int_{\mathbb R^n} |u|^{p+1}  = \sum_{i=1}^m \int_{\mathbb R^n} |\Delta u_i|^2< \infty. \end{equation} 
In this case,  if in addition $u$ is stable,  then each component $u_i$ must be identically zero.
\end{thm}

In order to prove the above theorem, we are required to establish some a priori estimates on solutions.  
\begin{lemma}\label{lemest}
Suppose that $u=(u_i)_{i=1}^m$ be a smooth stable solution of (\ref{main2}) and set $w=(w_i)_{i=1}^m$ where $w_i=\Delta u_i$. Then the following estimate holds,
\begin{equation}
\int_{\mathbb R^n} (|w|^2 + |u|^{p+1}) \zeta^2 \le C \int_{\mathbb R^n} |u|^2 \left( |\nabla \Delta \zeta| |\nabla \zeta| + |\Delta \zeta|^2 + |\Delta |\nabla \zeta|^2 |\right) +  |u| |w| |\nabla \zeta|^2 ,
\end{equation}
for a test function $\zeta :C_c^\infty(\mathbb R^n)\to\mathbb R$. 
\end{lemma}
\begin{proof}
Test the stability inequality on $u_i\zeta$ where $\zeta$ is a test function then multiply both sides of the $i^{th}$ equation of (\ref{main}) with $u_i \zeta^2$. Equating these completes the proof. 
\end{proof}

Applying an appropriate test function yields the following estimate. 
\begin{cor}\label{corest}
Under the same assumptions as in (\ref{lemest}), there exists a constant $C$ such that 
\begin{equation}
\int_{B_R(x)} |w|^2 +|u|^{p+1} \le C R^{-4} \int_{  B_{2R}(x)\setminus B_R(x)} |u|^2 + C R^{-2} \int_{  B_{2R}(x)\setminus B_R(x)} |u| |w|,
\end{equation}
and therefore 
\begin{equation}
\int_{B_R(x)} |w|^2+|u|^{p+1} \le C R^{n-4\frac{p+1}{p-1}}.
\end{equation}
\end{cor}
\begin{proof}
Set the test function $\zeta_R\in C^1_c(\mathbb{R}^n)$ in (\ref{lemest}) where $0\le\zeta_R\le1$ being the following test function
 $$\zeta_R(x)=\left\{
                      \begin{array}{ll}
                        1, & \hbox{if $|x|<R$,} \\
                        0, & \hbox{if $|x|>2R$,} 
                                                                       \end{array}
                    \right.$$
satisfying $||\nabla\zeta_R||_{\infty}\le R^{-1}$ and $||\Delta\zeta_R||_{\infty}\le R^{-2}$.

\end{proof}

We now provide classification of stable homogeneous solutions.  Note first that the following Hardy-Rellich inequality with the best constant holds. Suppose that $h: C_c^2(\mathbb R^n)\to\mathbb R$ then 
\begin{equation}\label{Hardy4}
\int_{\mathbb R^n} |\Delta  h|^2 dx \ge \frac{n^2(n-4)^2}{16} \int_{\mathbb R^n} \frac{h^2}{|x|^4}  dx. 
\end{equation} 
This inequality implies that the singular solution given by (\ref{singsol}) is stable if and only if  
\begin{equation}
 p |\mathcal A|^{p-1}= p \alpha(\alpha+2)(n-\alpha-2)(n-\alpha-4)\le  \frac{n^2(n-4)^2}{16}, \end{equation} 
where $\alpha:=\frac{4}{p-1}$ and $\mathcal A$ given in  (\ref{singsol}). 
\begin{thm}\label{thmhomo2}
Suppose that $u=(u_i)_{i=1}^m$ for $u_i= r^{-\frac{4}{p-1}} \psi_i(\theta)$ is a stable solution of (\ref{main2}). Then,  each $\psi_i\equiv 0$ provided $\frac{n+4}{n-4}<p<\bar p_c(n)$ where $\bar p_c(n)$ given by (\ref{pc2}). 
\end{thm}
\begin{proof} It is straightforward to see that $\psi=(\psi_i)_{i=1}^m$ satisfies 
\begin{equation}\label{deltatheta2psi}
\Delta_\theta^2 \psi_i - \alpha \Delta_\theta \psi_i+\beta \psi_i=|\psi|^{p-1} \psi_i ,
\end{equation}
for
 \begin{eqnarray}
\alpha := (q+2) (n-4-q) +q(n-2-q) \ \ \text{and}\ \ \beta := q(q+2) (n-4-q)(n-2-q),
\end{eqnarray}
where $q:=\frac{4}{p-1}$. Multiplying both sides of (\ref{deltatheta2psi}) with $\psi_i$ and integrating over $\mathcal S^{n-1}$,  we conclude  
\begin{equation}\label{eqpsi}
\sum_{i=1}^m \int_{\mathcal S^{n-1}} \left[ |\Delta \psi_i|^2 + \alpha |\nabla \psi_i|^2 \right]+\beta \int_{\mathcal S^{n-1}} |\psi|^2 =   \int_{\mathcal S^{n-1}} |\psi|^{p+1} . 
\end{equation}
We now test the stability inequality (\ref{stability}) on $\phi=(\phi_i)_{i=1}^m$ for $\phi_i:=r^{-\frac{n-4}{2}} \psi_i(\theta) \eta_\epsilon(r)$. Here, $\eta_\epsilon$ is a standard cut-off function $\eta_\epsilon\in C_c^1(\mathbb R_+)$ at the origin and at infinity that is $\eta_\epsilon=1$ for $\epsilon<r<\epsilon^{-1}$ and $\eta_\epsilon=0$ for either $r<\epsilon/2$ or $r>2/\epsilon$. Applying similar ideas provided in \cite{fw}, we get  
\begin{equation}\label{ineqpsi}
p \int_{\mathcal S^{n-1}} |\psi|^{p+1} \le \sum_{i=1}^m \int_{\mathcal S^{n-1}} \left[|\Delta \psi_i|^2   + \frac{n(n-4)}{2} |\nabla \psi_i|^2\right] + \frac{n^2(n-4)^2}{16} \int_{\mathcal S^{n-1}} |\psi|^2 .
\end{equation}
Combining (\ref{eqpsi}) and (\ref{ineqpsi}) we get 
\begin{equation}
(p-1) \sum_{i=1}^m \int_{\mathcal S^{n-1}} \left[|\Delta \psi_i|^2 + \left(p\alpha- \frac{n(n-4)}{2} \right) |\nabla \psi_i|^2 \right] + \left(p\beta- \frac{n^2(n-4)^2}{16} \right) |\psi|^2 \le 0 .
\end{equation}
Note the coefficients $p-1, p\alpha- \frac{n(n-4)}{2} $ and $p\beta- \frac{n^2(n-4)^2}{16} $ are positive when $\frac{n+4}{n-4} < p < p_c(n)$ where $p_c(n)$ is given by (\ref{pc2}). 
\end{proof}

\noindent {\bf Proof of Theorem \ref{thm2stab}}. The proof of the Sobolev critical case relies on applying the Pohozaev identity and we omit it here. We now provide a sketch of the proof when $p>\frac{n+4}{n-4}$  in a few steps.  
\\
Step 1. $\lim_{r\to\infty} E(u,r,0)<\infty$.   Note that the energy $E(u,r,0)$ is nondecreasing in $r$, as given in Theorem \ref{}. Therefore, 
\begin{equation}\label{Ebound}
E(u,r,0) \le r^{-1} \int_r ^{2r} E(u,t,0) dt \le r^{-2}\int_r ^{2r} \int_t ^{t+r} E(u,\lambda,0) d \lambda dt. 
\end{equation}
Applying estimates given in Corollary \ref{corest}  imply that  the right-hand side of (\ref{Ebound}) is bounded.  
\\
Step 2. Define $u_i^\lambda (x)= \lambda^{\frac{4}{p-1}} u_i(\lambda x)$ for each $1\le i\le m$ where $u=(u_i)_{i=1}^m $ is a stable solution of (\ref{main}).  Then $u_i^\lambda\to u_i^\infty$ where $u_i^\infty\in W^{2,2}_{loc}(\mathbb R^n) \cap L^{p+1}_{loc}(\mathbb R^n)$ and $u^\infty=(u^\infty_i)_{i=1}^m$ is a stable solution of (\ref{main}).    Set $w_i^\lambda(x):=\lambda^{\frac{4}{p-1}+2} w_i(\lambda x)$. From Corollary \ref{corest}, we have 
\begin{equation}
 \int_{B_r(x)} |w^\lambda|^2+|u^\lambda|^{p+1} \le C r^{n-4\frac{p+1}{p-1}}.
\end{equation}
From elliptic estimates, up to a subsequence, $u_i^\lambda\to u^\infty_i$ for each $1\le i\le m$ weakly in $W^{2,2}_{loc}(\mathbb R^n)\cap L^{p+1}_{loc}(\mathbb R^n)$. From compactness embeddings and applying interpolation we arrive at $u_i^\lambda \to u_i^\infty$ in $L^q_{loc}(\mathbb R^n)$ for any $q\in [1,p+1)$.  Note also that $u^\infty$ is a stable solution, since $u^\lambda$ is a stable solution and we can send $\lambda $ to infinity.  
\\
Step 3. $u^\infty$ is a homogeneous solution.  This is a direct consequence of the monotonicity formula and the following fact 
\begin{equation}\label{limE}
\lim_{\lambda\to \infty} [ E(u^\lambda, R , 0) - E(u^\lambda, r , 0) ]=0. 
\end{equation}
The left-hand side of (\ref{limE}) is bounded from below by   
\begin{eqnarray}
 E(u^\lambda, R , 0) - E(u^\lambda, r , 0)  
 &\ge& C(n,p)  \sum_{i=1}^m    \int_{   B_R \setminus B_r }  \left(  \frac{4}{p-1} |x|^{-1} u^\lambda_i+ \frac{\partial u^\lambda_i}{\partial r}\right)^2 |x|^{\frac{8}{p-1}+2-n} dx 
 \\&=& C(n,p)  \sum_{i=1}^m    \int_{   B_R \setminus B_r }  \left(  \frac{4}{p-1} |x|^{-1} u^\infty_i+ \frac{\partial u^\infty_i}{\partial r}\right)^2 |x|^{\frac{8}{p-1}+2-n} dx . 
\end{eqnarray}
This implies that for each $1\le i\le m$, we have
\begin{equation}
u_i^\infty(x)=|x|^{-\frac{4}{p-1}} u^\infty_i \left(  \frac{x}{|x|}\right). 
\end{equation}
This completes the proof of this step. 
\\
Step 4. $\lim_{r\to\infty} E(u,r,0)=0$.  Since each $u_i^\infty$ is a homogenous function,  Theorem \ref{thmhomo2} implies that $u^\infty =0$. Therefore, $\lim_{\lambda \to \infty} u_i^\lambda =0$ strongly in $L^2(B_4)$ for each $i=1,\cdots,m$ that is 
 \begin{equation}
\lim_{\lambda \to \infty} \int_{B_4} |u_i^\lambda|^2 =0 \ \ \text{and }\ \   \lim_{\lambda \to \infty} \int_{B_4} |u_i^\lambda w_i^\lambda| =0.
\end{equation}
Applying  Corollary \ref{corest},  we conclude  
\begin{equation}
 \lim_{\lambda \to \infty} \int_{B_4} |w^\lambda|^2 +  |u^\lambda|^{p+1}  =0.
\end{equation}
On the other hand, there exists $r_0>0$ such that 
\begin{equation} \lim _{i\to \infty} ||u^{\lambda_i}||_{W^{2,2}(\partial B_{r_0})}=0.\end{equation}
This implies that 
\begin{equation} \lim _{i\to \infty} E(u,\lambda_i r_0,0) =  \lim _{i\to \infty} E(u^{\lambda_i},r_0,0) .\end{equation}
The fact that $E$ is nondecreasing, that is given as a monotonicity formula in Theorem \ref{thmonos2},   completes the proof. 

\hfill $ \Box$

 \section{Nonlocal Case: Lower and higher order fractional Laplacian}\label{nls}
 
 In this section, we consider system (\ref{main}) with the fractional Laplacian operator $(-\Delta)^s$ where $s\in(0,2)$ for $s\neq1$ and we establish Theorem \ref{mainthm}.   We first note that the following Hardy inequality holds for $n>2s$
\begin{equation}\label{Hardys}
\int_{\mathbb R^n} |\xi|^{2s} |\hat h|^2 d\xi > \Lambda_{n,s} \int_{\mathbb R^n} |x|^{-2s} h^2 dx,
\end{equation}
for any $h \in C_c^\infty(\mathbb R^n)$ where the optimal constant is given by 
\begin{equation}
 \Lambda_{n,s}=2^{2s}\frac{ \Gamma(\frac{n+2s}{4})^2  }{ \Gamma(\frac{n-2s}{4})^2}.
 \end{equation}
 Note that $\Lambda_{n,2}=\frac{n^2(n-4)^2}{16} $  meaning  (\ref{Hardys}) recovers (\ref{Hardy4}) for $s=2$.  For more information interested readers are encouraged to see  \cite{h} by Herbst  (and also \cite{ya}).  We now provide a classification result for homogeneous solutions. Note that for the case of scalar equations, that is when $m=1$, this classification was given in \cite{fw}.  Note also that the proof is valid  regardless of magnitude of the parameter $s$. 
   \begin{thm}\label{homog}
 Suppose  $u_i=r^{-\frac{2s}{p-1}} \psi_i(\theta)$ is a stable solution of (\ref{main}) for  $s\in(0,2)$ and $s\neq1$. Then,  each $\psi_i$ vanishes identically, provided
 $p>\frac{n+2s}{n-2s}$ and 
\begin{equation}
 p \frac{\Gamma(\frac{n}{2}-\frac{s}{p-1}) \Gamma(s+\frac{s}{p-1})}{\Gamma(\frac{s}{p-1}) \Gamma(\frac{n-2s}{2}-\frac{s}{p-1})}>\frac{ \Gamma(\frac{n+2s}{4})^2 }{\Gamma(\frac{n-2s}{4})^2} .  \end{equation}
 \end{thm}
 \begin{proof}
 Since $u=(u_i)_{i=1}^m$ is a  solution of (\ref{main}),  each $\psi_i$ satisfies 
 \begin{eqnarray}
\ \ \  |x|^{-\frac{2ps}{p-1}} |\psi|^{p-1}(\theta) \psi_i(\theta)
&=& \int \frac{  |x|^{-\frac{2s}{p-1}} \psi_i(\theta) -|y|^{-\frac{2s}{p-1}} \psi_i(\sigma) }{ |x-y|^{n+2s}} dy 
\\ &=& \nonumber  |x|^{-\frac{2ps}{p-1}}  [  \int \frac{  \psi_i(\theta) - t^{-\frac{2s}{p-1}} \psi_i(\theta) }{ (t^2+1- 2t <\theta,\sigma>)^{\frac{n+2s}{2}} }  t^{n-1} dt d\sigma 
\\&&+ \nonumber \int \frac{  t^{-\frac{2s}{p-1}} (\psi_i(\theta) -  \psi_i(\sigma) }{ (t^2+1- 2t <\theta,\sigma>)^{\frac{n+2s}{2}} }  t^{n-1} dt d\sigma]  , 
\end{eqnarray}
where we have used the change of variable $|y|=t|x|$. Simplifying the above, for each $i$,  we obtain 
 \begin{equation}\label{AnsEq}
 \psi_i(\theta) A_{n,s} + \int_{\mathbb S^{n-1}} K_{\frac{2s}{p-1}} (<\theta,\sigma>) (\psi_i(\theta)-\psi_i(\sigma)) d \sigma= |\psi|^{p-1}(\theta) \psi_i(\theta) ,  
 \end{equation}
 where 
  \begin{equation}\label{Ans}
A_{n,s}:=\int_0^\infty \int_{\mathbb S^{n-1}} \frac{1-t^{  -\frac{2s}{p-1} }}{    (t^2+1- 2t <\theta,\sigma>)^{\frac{n+2s}{2}} } t^{n-1} d\sigma dt, 
 \end{equation}
 and 
   \begin{equation}\label{Ksp}
  K_{\frac{2s}{p-1}}(<\theta,\sigma>):=\int_0^\infty \frac{t^{  n-1-\frac{2s}{p-1} }}{    (t^2+1- 2t <\theta,\sigma>)^{\frac{n+2s}{2}} }  dt . 
 \end{equation}
 Multiplying  (\ref{AnsEq}) with $\psi_i$ and integrating we get 
 \begin{equation}\label{Ans2}
\int_{\mathbb S^{n-1}} |\psi|^2(\theta) A_{n,s} + \int_{\mathbb S^{n-1}} K_{\frac{2s}{p-1}} (<\theta,\sigma>) |\psi(\theta)-\psi(\sigma)|^2 d\theta d \sigma= \int_{\mathbb S^{n-1}} |\psi|^{p+1}(\theta) d\theta. 
  \end{equation}
We now test the stability inequality (\ref{stability})  for $\phi_i (x)=r^{-\frac{n-2s}{2}} \psi _i (\theta) \eta_\epsilon(r)$ and $u_i=r^{-\frac{2s}{p-1}} \psi_i(\theta)$ with the same  $\eta_\epsilon(r)$ as the one given in the proof of Theorem \ref{thmhomo2}.  Applying similar ideas provided in \cite{fw}, we conclude 
\begin{equation}\label{lambdansin}
  \Lambda_{n,s} \int_{\mathbb S^{n-1}} |\psi|^2 + \int_{\mathbb S^{n-1}} K_{\frac{n-2s}{2}}(<\theta,\sigma>) |\psi(\theta)-\psi(\sigma)|^2 d\sigma \ge p  \int_{\mathbb S^{n-1}} |\psi|^{p+1} , 
 \end{equation}
when  
\begin{equation}
\Lambda_{n,s} :=\int_0^\infty  \int_{\mathbb S^{n-1}}\frac{1-t^{  \frac{n-2s}{2} }}{    (t^2+1- 2t <\theta,\sigma>)^{\frac{n+2s}{2}} }  t^{n-1} d\sigma dt.
\end{equation}
Combining (\ref{lambdansin})  and (\ref{Ans2}), we end up with 
  \begin{equation}
  ( \Lambda_{n,s} -p A_{n,s}) \int_{\mathbb S^{n-1}} |\psi|^2 + \int_{\mathbb S^{n-1}} (K_{\frac{n-2s}{2}} - pK_{\frac{2s}{p-1}}  )(<\theta,\sigma>) |\psi(\theta)-\psi(\sigma)|^2 d\sigma \ge 0 . 
  \end{equation}
The fact that $K_\alpha$ is decreasing in $\alpha$ implies $K_{\frac{n-2s}{2}} < K_{\frac{2s}{p-1}}$ for $p>\frac{n+2s}{n-2s}$. Therefore,  $K_{\frac{n-2s}{2}} - pK_{\frac{2s}{p-1}} <0$. On the other hand the assumption of the theorem implies that $\Lambda_{n,s} -p A_{n,s}<0$. Therefore, each $\psi_i$ vanishes identically. This completes the proof. 

 \end{proof}

  \subsection{Lower Order Fractional Laplacian Operator}
  In this part,  we show that Theorem \ref{mainthm} holds when $0<s<1$.   To do so we provide the following estimate first. 
 \begin{lemma}\label{lemv}
Suppose that $p\neq \frac{n+2s}{n-2s}$. Let $u$ be a solution of (\ref{main}) that is stable outside a ball $B_{R_0}$ and $v$ satisfies (\ref{main1e}). Then there exists a constant $C>0$ such that
  \begin{equation}
 \int_{B_R} y^{1-2s} |v|^2 \le C R^{n+2-2s \frac{p+1}{p-1}}, 
   \end{equation}
 for any $R>3R_0$.
\end{lemma} 
 \begin{lemma}\label{lemvdeltav}
  Let $u$ be a solution of (\ref{main}) that is stable outside a ball $B_{R_0}$ and $v$ satisfies (\ref{main1e}). Then there exists a positive constant $C$  such that
 \begin{equation}
\int_{B_R\cap\partial\mathbb R_+^{n+1}}  |v|^{p+1} +  \sum_{i=1}^m \int_{B_R\cap\mathbb R_+^{n+1}}  y^{1-2s} |\nabla v_i|^2   \le C R^{n-2s \frac{p+1}{p-1}}. 
  \end{equation}
\end{lemma}
 
  \noindent {\bf Proof of Theorem \ref{mainthm} when $0<s<1$}.
  We omit the proof here since arguments are very similar to the ones which will be provided for the  case of $1<s<2$. 
  
   \hfill $ \Box$

  \subsection{Higher Order Fractional Laplacian Operator}
  
   As the last  past of this section,  we shall restrict ourselves to the case $1<s<2$.  Let us start with the following integral estimate on stable solutions. 
  \begin{lemma}\label{estimate}  Let $u$ be a solution of (\ref{main}) that is stable outside a ball $B_{R_0}$ and $v$ satisfies (\ref{main2e}). Then there exists a positive constant $C$  such that
  \begin{eqnarray}\label{}
 \int_{\partial \mathbb R_+^{n+1}}  |v|^{p+1} \eta^2 + \sum_{i=1}^m \int_{ \mathbb R_+^{n+1}}  y^b |\Delta_b v_i|^2 \eta^2 
&\le& C \int_{\mathbb R_+^{n+1}} y^b |v|^2 \left( |\Delta_b \eta|^2 +|\Delta_b |\nabla \eta|^2|+|\nabla\eta\cdot\nabla\Delta_b\eta|\right)  \\&&+ C \sum_{i=1}^m \int_{\mathbb R_+^{n+1}} y^b  |v_i| |\Delta_b v_i| |\nabla \eta|^2. 
 \end{eqnarray}
 \end{lemma}
 \begin{proof} Here is a sketch of the proof. Multiply the $i^{th}$ equation of (\ref{main2e}) with $y^b u_i\eta^2$ where $\eta$ is a test function to get 
  \begin{equation}\label{pdedelta}
C_{n,s} \int_{ \partial\mathbb R_+^{n+1}} |v|^{p+1} \eta^2= \sum_{i=1}^m \int_{ \mathbb R_+^{n+1}} y^b \Delta_b v_i \Delta_b(v_i\eta^2). 
 \end{equation}
 Note also that for each $i$ these identities hold for a test function $\eta$
 \begin{eqnarray}
 \Delta_b v_i \Delta_b(v_i \eta^2)- |\Delta_b(v_i\eta)|^2 
 &=& -v_i^2 |\Delta_b \eta|^2 +2v_i \Delta_b v_i |\nabla \eta|^2 -4 |\nabla v_i \cdot\nabla \eta|^2 
  -4 v_i \Delta_b\eta \nabla v_i \cdot\nabla \eta,
 \\ \Delta_b (v_i \eta) &=& \eta \Delta_b v_i  + v_i \Delta_b \eta +2\nabla v_i \cdot\nabla\eta . 
 \end{eqnarray}
 Applying these identifies together with (\ref{pdedelta}) one can see that 
 \begin{eqnarray}\label{vp}
 C_{n,s} \int_{ \partial\mathbb R_+^{n+1}} |v|^{p+1} \eta^2 &=&  \sum_{i=1}^m \int_{ \mathbb R_+^{n+1}} y^b |\Delta_b(v_i\eta)|^2   +2\sum_{i=1}^m \int_{ \mathbb R_+^{n+1}} y^b v_i \Delta_b v_i |\nabla \eta|^2 
\\&&\nonumber  -4 \sum_{i=1}^m \int_{ \mathbb R_+^{n+1}} y^b  |\nabla v_i \cdot\nabla \eta|^2 + \int_{ \mathbb R_+^{n+1}} y^b |v|^2 (|\Delta_b \eta|^2 + 2\nabla\eta\cdot\nabla\Delta_b \eta ) .
 \end{eqnarray}
 Testing the stability inequality (\ref{stability}) on $\phi_i=u_i\eta$ and applying (\ref{vp}) one can finish the proof. 
  \end{proof}
 One can set $\eta$ to be the standard test function to get the following estimate. 
  
\begin{cor}\label{ue2} With the same assumption as Lemma \ref{estimate}. Then there exists a positive constant $C$  such that
  \begin{equation}
\int_{B_R\cap\partial\mathbb R_+^{n+1}}  |v|^{p+1} + \sum_{i=1}^m \int_{B_R\cap\mathbb R_+^{n+1}}  y^b |\Delta_b v_i|^2  \le C R^{-4} \int_{B_R\cap\mathbb R_+^{n+1}} y^b |v|^2 . 
  \end{equation}
\end{cor}
Here we provide more decay estimates of solutions. These lemmata are main tools in our proof of Theorem \ref{mainthm}.  We omit the proofs here and we refer interested readers to see the proof of Lemma 4.5-4.6 in \cite{fw} where similar arguments are applied. 

\begin{lemma}\label{lemv}
Suppose that $p\neq \frac{n+2s}{n-2s}$. Let $u$ be a solution of (\ref{main}) that is stable outside a ball $B_{R_0}$ and $v$ satisfies (\ref{main2e}). Then there exists a constant $C>0$ such that
  \begin{equation}
 \int_{B_R} y^b |v|^2 \le C R^{n+4-2s \frac{p+1}{p-1}}, 
   \end{equation}
 for any $R>3R_0$.
\end{lemma} 
 \begin{lemma}\label{lemvdeltav}
  Let $u$ be a solution of (\ref{main}) that is stable outside a ball $B_{R_0}$ and $v$ satisfies (\ref{main2e}). Then there exists a positive constant $C$  such that
 \begin{equation}
\int_{B_R\cap\partial\mathbb R_+^{n+1}}  |v|^{p+1} + \sum_{i=1}^m  \int_{B_R\cap\mathbb R_+^{n+1}}  y^b |\Delta_b v_i|^2   \le C R^{n-2s \frac{p+1}{p-1}}. 
  \end{equation}
\end{lemma}

 \noindent {\bf Proof of Theorem \ref{mainthm} when $1<s<2$}. The proof of the case $p=p_S(n,s)$ is based on the Pohozaev identity provided in \cite{fw,rs},  and we omit it here. The proof is based on the monotonicity formula that is Theorem \ref{thmono2s} and a blow-down analysis. 
 \\
 Step 1.  The energy is bounded, that is $\lim_{\lambda\to \infty} E(v,0,\lambda)<\infty$.  This is a direct consequence of the monotonicity formula and Lemma \ref{lemv} and Lemma \ref{lemvdeltav}. Similar to the proof of Theorem \ref{thm2stab} we have
 \begin{equation}\label{E2bound}
E(v,\lambda,0) \le \lambda^{-2} \int_\lambda ^{2\lambda} \int_{t}^{t+\lambda} E(v,\gamma,0) d\gamma dt .
\end{equation}
 Lemma \ref{lemvdeltav} and Lemma \ref{lemv} imply that the right-hand side of (\ref{E2bound}) is bounded. 
 \\
 Step 2. The sequence  $v_i^{\lambda}$ converges weakly in $H^1_{loc}(\mathbb R^n, y^{3-2s} dxdy)$ to a function $v_i^\infty$ where each $v_i^\infty$ is homogeneous for $1\le i\le m$ and therefore they are zero.    Note that the convergence part is a direct consequence of the elliptic estimates. We now show that each $v_i^\infty$ is homogeneous. From the boundedness of the energy we have 
\begin{equation}
  \lim_{k\to\infty} \left[ E(v,R_2\lambda_k,0)- E(v,R_1\lambda_k,0) \right] =0. 
\end{equation} 
From this and applying the scaling invariance of the energy and also the monotonicity formula we get 
\begin{eqnarray}
 0&=&\liminf_{k\to\infty} \sum_{i=1}^m \int_{(B_{R_2} \setminus B_{R_1})\cap \mathbb R^{n+1}_+} y^{3-2s}  r^{\frac{4s}{p-1}+2s-2-n}   \left(  \frac{2s}{p-1} r^{-1} v_i^{\lambda_k}+ \frac{\partial v_i^{\lambda_k}}{\partial r}\right)^2 dy dx
\\&\ge&  \sum_{i=1}^m \int_{(B_{R_2} \setminus B_{R_1})\cap \mathbb R^{n+1}_+} y^{3-2s}  r^{\frac{4s}{p-1}+2s-2-n}   \left(  \frac{2s}{p-1} r^{-1} v_i^{\infty}+ \frac{\partial v_i^{\infty}}{\partial r}\right)^2 dy dx ,
   \end{eqnarray}
since we have the weak convergence of $(v_i^{\lambda_k})$ to $v_i^\infty$ in $H^1_{loc}(\mathbb R^n,y^{3-2s} dydx)$. Therefore, 
\begin{equation}
 \frac{2s}{p-1} r^{-1} v_i^{\infty}+ \frac{\partial v_i^{\infty}}{\partial r}=0 \ \ \text{a.e. \ \ in} \ \ \mathbb R_+^{n+1},
 \end{equation}
for each $1\le i\le m$. 
\\
Step 3.  $\lim_{\lambda\to\infty} E(v,\lambda,0)=0$. Note that the monotonicity formula implies that 
\begin{equation}
E(v,\lambda,0) \le \lambda^{-1}  \int_{\lambda}^{2\lambda} E(t) dt \le \sup_{[\lambda,2\lambda]} I + C \lambda^{-n-1+\frac{2s(p+1)}{p-1}} \int_{B_{2\lambda}\setminus B_\lambda} |v|^2 ,
\end{equation}
where 
 \begin{equation}\label{}
I(v,\lambda)= I(v^\lambda,1) =\frac{1}{2} \sum_{i=1}^m \int_{\r\cap B_{1}} y^{3-2s}\vert\Delta_b v_i^\lambda\vert^2  dxdy
 -\frac{\kappa_{s}}{p+1} \int_{\br\cap B_{1}} \vert v^\lambda\vert^{p+1}dx .
 \end{equation}
Note that $\lim \lambda\to\infty I(v,\lambda)=0$. On the other hand,  from the fact that $u^\infty$ is a homogenous solution we have 
\begin{equation}
\lim_{\lambda \to \infty} u_i^\lambda =0  ,
\end{equation}
strongly in $L^2(B_4)$. Therefore, 
 \begin{equation}
\lim_{\lambda \to \infty} \int_{B_4} |u^\lambda|^2 =0. 
\end{equation} 
 This implies that  $\lim_{\lambda\to\infty} E(v,\lambda,0)=0$ and completes the proof.

 \hfill $ \Box$
 
 We end this section with the following open problem in regards to monotonicity formulae for the Lane-Eden system.  
\begin{op}
Consider the Lane-Emden system for any parameters $s>0 $ and $1<q<p$
 \begin{eqnarray}\label{laneemdensystem}
 \left\{ \begin{array}{lcl}
\hfill (- \Delta)^s  u &=& v^p   \ \ \text{in}\ \ \mathbb{R}^n,\\   
\hfill (- \Delta)^s v &=& u^q    \ \ \text{in}\ \ \mathbb{R}^n.
\end{array}\right.
  \end{eqnarray}
Proving a monotonicity formula for the above system, similar to the ones given in Theorem \ref{thmonos1}-\ref{thmono2s},  seems more challenging to derive. Note that for the case of $1<q=p$ and $1=q<p$ such monotonicity formulae are known, see the introduction.   Needless to mention  that the Lane-Emden system (\ref{laneemdensystem}) is not a gradient system,  meaning that it is not of the following form  \begin{equation}\label{H}
  (-\Delta)^s u= \nabla H(u)   \quad  \text{in} \ \  \mathbb{R}^n. 
  \end{equation}
where $u:\mathbb R^n\to\mathbb R^m$. 
\end{op} 

  Lastly,  let us mention that in \cite{fsh} we apply monotonicity formulae derived in this article  to prove regularity of free boundaries and partial regularity of weak solutions for certain coupled elliptic systems.

  \section{Acknowledgement.} The first author appreciates H. Shahgholian and J. Wei for their hospitality during his visits to KTH and UBC,  and he is thankful to E. Hebey for his talk at the conference on the occasion of Michael Struwe's 60th birthday at ETH.

\end{document}